\documentclass[12pt]{amsart}

\usepackage{amscd,amssymb,amsopn,amsmath,amsthm,mathrsfs,graphics,amsfonts,enumerate,verbatim,calc
}

\usepackage{color}

\usepackage{systeme}
\usepackage{graphicx,color}

\usepackage[OT2,OT1]{fontenc}
\newcommand\cyr{%
\renewcommand\rmdefault{wncyr}%
\renewcommand\sfdefault{wncyss}%
\renewcommand\encodingdefault{OT2}%
\normalfont
\selectfont}
\DeclareTextFontCommand{\textcyr}{\cyr}

\usepackage{amssymb,amsmath}

\DeclareFontFamily{OT1}{rsfs}{}
\DeclareFontShape{OT1}{rsfs}{n}{it}{<-> rsfs10}{}
\DeclareMathAlphabet{\mathscr}{OT1}{rsfs}{n}{it}

\numberwithin{equation}{section}
\newtheorem{theorem}{Theorem}[section]
\newtheorem{lemma}[theorem]{Lemma}

\newtheorem{corollary}[theorem]{Corollary}

\newtheorem{question}{Question}

\theoremstyle{definition}
\newtheorem{definition}[theorem]{Definition}
\newtheorem{remark}[theorem]{Remark}
\theoremstyle{remark}

\newtheorem{example}[theorem]{Example}

\newcommand{\Ext}{\operatorname{Ext}}

\newcommand{\fm}{\frak{m}}

\newcommand{\fq}{\frak{q}}

\newcommand{\fn}{\frak{n}}

\begin{document}
\title[Upper bound of multiplicity]{Upper bound of multiplicity in Cohen-Macaulay rings of prime characteristic}

\author[Duong Thi Huong]{Duong Thi Huong}
\address{Department of Mathematics, Thang Long University, Hanoi, Vietnam}
\email{huongdt@thanglong.edu.vn}

\author[Pham Hung Quy]{Pham Hung Quy}
\address{Department of Mathematics, FPT University, Hanoi, Vietnam}
\email{quyph@fe.edu.vn}


\begin{abstract} Let $(R, \fm)$ be a local ring of prime characteristic $p$ and of dimension $d$ with the embedding dimension $v$, type $s$ and the Frobenius test exponent for parameter ideals $\mathrm{Fte}(R)$. We will give an upper bound for the multiplicity of Cohen-Macaulay rings in prime characteristic in terms of $\mathrm{Fte}(R),d,v$ and $s$. Our result extends the main results for Gorenstein rings due to Huneke and Watanabe \cite{HW15}.
\end{abstract}

\maketitle

\section{Introduction}
Throughout this paper, let $(R, \fm)$ be a Noetherian commutative local ring of prime characteristic $p$ and of dimension $d$. In 2015, Huneke and Watanabe \cite{HW15} gave an upper bound of the multiplicity $e(R)$ of an $F$-pure ring $R$ in terms of the dimension $d$ and the embedding dimension $v$. Namely, Huneke and Watanabe proved that 
$$e(R) \le \binom{v}{d}$$
for any $F$-pure ring. If $R$ is $F$-rational, the authors of \cite{HW15} provided a better bound that $e(R) \le \binom{v-1}{d-1}$ (cf. \cite[Theorem 3.1]{HW15}). 
If $R$ is Gorenstein, the upper bound is largely reduced by the duality as follows (cf. \cite[Theorem 5.1]{HW15})\\
(1) If $R$ is Gorenstein and $F$-pure then 
$$e(R)\leq \systeme*{{2\binom{v-r-1}{r}}\quad\quad\quad  \text{ if }\dim (R)=2r+1,
{\binom{v-r}{r}+\binom{v-r-1}{r-1}} \text{ if }\dim (R)=2r.
}
$$
(2) If $R$ is Gorenstein and $F$-rational then
$$e(R)\leq\systeme*{{{{v-r-1}\choose{r}}+{{v-r-2}\choose{r-1}}} \text{ if }\dim (R)=2r+1,
{2 {{v-r-1}\choose{r-1}}}\quad\quad\quad \text{ if }\dim (R)=2r.
}
$$
In 2019, Katzman and Zhang tried to remove the $F$-pure condition in Huneke-Watanabe's theorem by using the Hartshorne-Speiser-Lyubeznik number $\mathrm{HSL}(R)$. Notice that $\mathrm{HSL}(R) = 0$ if $R$ is $F$-injective (e.g. $R$ is $F$-pure). If $R$ is Cohen-Macaulay, Katzman and Zhang \cite[Theorem 3.1]{KZ18} proved the following inequality
$$e(R) \le Q^{v-d} \binom{v}{d},$$
where $Q = p^{\mathrm{HSL}(R)}$. They also constructed examples to show that their bound is asymptotically sharp (cf. \cite[Remark 3.2]{KZ18}). Recall that the Frobenius test exponent for parameter ideals of $R$, denoted by $\mathrm{Fte}(R)$, is the least integer (if exists) $e$ satisfying that $(\fq^F)^{[p^e]} = \fq^{[p^e]}$ for every parameter ideal $\fq$, where $\fq^F$ is the Frobenius closure of $\fq$. It is asked by Katzman and Sharp that whether $\mathrm{Fte}(R) < \infty$ for every (equidimensional) local ring (cf. \cite{KS06}). If $R$ is Cohen-Macaulay then $\mathrm{Fte}(R) = \mathrm{HSL}(R)$. Moreover the question of Katzman and Sharp has affirmative answers when $R$ is either generalized Cohen-Macaulay by \cite{HKSY06} or $F$-nilpotent by \cite{Q19} (see the next section for the details). Recently, we (cf. \cite[Theorem 3.]{HQ20}) extended the result for any ring of finite Frobenius test exponent for parameter ideals. Set $Q = p^{\mathrm{Fte}(R)}$, we have:\\
(1) Suppose $\mathrm{Fte}(R) < \infty$. Then 
$$e(R) \le Q^{v-d}\binom{v}{d}.$$ 
(2) If $R$ is $F$-nilpotent then 
$$e(R) \le Q^{v-d}\binom{v-1}{d-1}.$$ 
The main result of this paper is given reduced upper bound for the multiplicity of the ring when the ring is Cohen-Macaulay, that is a natural extension of Huneke and Watanabe's result in Gorenstein cases (cf. \cite[Theorem 5.1]{HW15}).
\begin{theorem}
 Let $(R, \fm)$ be a Cohen-Macaulay local ring of prime characteristic $p$ with the dimension $d$, the embedding dimension $v$ and the type $s$. Set $Q = p^{\mathrm{Fte}(R)}$. Then
\begin{enumerate}[(1)]
\item We have
$$e(R)\leq \systeme*{{(s+1)Q^{v-d}\binom{v-r-1}{r}}\quad\quad\quad  \text{ if }\dim (R)=2r+1,
{\frac{(s+1)}{2} Q^{v-d}\left(\binom{v-r}{r}+\binom{v-r-1}{r-1}\right)} \text{ if }\dim (R)=2r.
}
$$
\item If $R$ is $F$-nilpotent then 
$$e(R)\leq\systeme*{{\frac{(s+1)}{2} Q^{v-d}\left({{v-r-1}\choose{r}}+{{v-r-2}\choose{r-1}}\right)} \text{ if }\dim (R)=2r+1 ,
{(s+1)Q^{v-d}{{v-r-1}\choose{r-1}}}\quad\quad\quad \text{ if }\dim (R)=2r.
}
$$
\end{enumerate}
\end{theorem}

We will prove the above theorem in the last section. In the next section we collect some useful materials.
\section{Preliminaries}
We firstly give the definition of the tight closure and the Frobenius closure of ideals.
\begin{definition}[\cite{HH90,H96}]
Let $R$ have characteristic $p$. We denote by $R^{\circ}$ the set of elements
of $R$ that are not contained in any minimal prime ideal. Then for any ideal $I$ of $R$ we define
\begin{enumerate}
  \item The {\it Frobenius closure} of $I$, $I^F = \{x \mid  x^{p^e} \in I^{[p^e]} \text{ for some } p^e\}$, where $I^{[p^e]} = (x^{p^e} \mid x \in I)$.
  \item The {\it tight closure} of $I$, $I^* = \{x \mid cx^{p^e} \in I^{[p^e]} \text{ for some } c \in R^{\circ} \text{ and for all } p^e \gg 0\}$.
\end{enumerate}
\end{definition}
The Frobenius endomorphism of $R$ induces the natural Frobenius action on local cohomology $F: H^i_{\fm}(R) \to H^i_{\fm}(R)$ for all $i \ge 0$. By a similar way, we can define the Frobenius closure and tight closure of zero submodule of local cohomology, and denote by $0^F_{H^i_{\fm}(R)}$ and $0^*_{H^i_{\fm}(R)}$ respectively.

Let $I$ be an ideal of $R$. The {\it Frobenius test exponent} of $I$, denoted by $\mathrm{Fte}(I)$, is the smallest number $e$ satisfying that $(I^F)^{[p^e]} = I^{[p^e]}$. By the Noetherianess of $R$, $\mathrm{Fte}(I)$ exists (and depends on $I$). In general, there is no upper bound for the Frobenius test exponents of all ideals in a local ring by the example of Brenner \cite{B06}. In contrast, Katzman and Sharp \cite{KS06} showed the existence of a uniform bound of Frobenius test exponents if we restrict to the class of parameter ideals in a Cohen-Macaulay local ring. For any local ring $(R, \fm)$ of prime characteristic $p$ we define the {\it Frobenius test exponent for parameter ideals}, denoted by $\mathrm{Fte}(R)$, is the smallest integer $e$ such that $(\fq^F)^{[p^e]} = \fq^{[p^e]}$ for every parameter ideal $\fq$ of $R$, and $\mathrm{Fte}(R) = \infty$ if we have no such integer. Katzman and Sharp raised the following question.
\begin{question}\label{Question}
Is $\mathrm{Fte}(R)$ a finite number for any (equidimensional) local ring?
\end{question}
The Frobenius test exponent for parameter ideals is closely related to an invariant defined by the Frobenius actions on the local cohomology modules $H^i_{\fm}(R)$, namely {\it the Hartshorne-Speiser-Lyubeznik number} of $H^i_{\fm}(R)$. The Hartshorne-Speiser-Lyubeznik number of $H^i_{\fm}(R)$ is a nilpotency index of Frobenius action on $H^i_{\fm}(R)$ and it is defined as follows 
$$\mathrm{HSL}(H^i_{\fm}(R)) = \min \{e \mid F^e(0^F_{H^i_{\fm}(R)}) = 0  \}.$$
By \cite[Proposition 1.11]{HS77} and \cite[Proposition 4.4]{L97}, $\mathrm{HSL}(H^i_{\fm}(R))$ is well defined (see also \cite{Sh07}). The Hartshorne-Speiser-Lyubeznik number of $R$ is 
$$\mathrm{HSL}(R) = \max \{\mathrm{HSL}(H^i_{\fm}(R)) \mid i = 0, \ldots, d\}.$$
\begin{remark}\label{Fte}
\begin{enumerate}
\item If $R$ is Cohen-Macaulay then $\mathrm{Fte}(R) =\mathrm{HSL}(R)$ by Katzman and Sharp \cite{KS06}. In general, we proved in \cite{HQ19} that $\mathrm{Fte}(R) \ge \mathrm{HSL}(R)$. Moreover, Shimomoto and Quy \cite[Main Theorem B]{QS17} constructed a local ring satisfying that $\mathrm{HSL}(R) = 0$, i.e. $R$ is $F$-injective, but $\mathrm{Fte}(R) > 0$.
\item Huneke, Katzman, Sharp and Yao \cite{HKSY06} gave an affirmative answer for Question \ref{Question} for generalized Cohen-Macaulay rings. 
\item In 2019, Quy \cite{Q19} provided a simple proof for the theorem of Huneke, Katzman, Sharp and Yao. By the same method he also proved that $\mathrm{Fte}(R) < \infty$ if $R$ is $F$-nilpotent. In 2019, Maddox \cite{M18} extended this result for {\it generalized $F$-nilpotent} rings (i.e. $H^i_{\fm}(R) / 0^F_{H^i_{\fm}(R)}$ has finite length for all $i < d$), and more general in \cite{HQ21} by us.
\end{enumerate}
\end{remark} 
We next recall some classes of $F$-singularities mentioned in this paper.
\begin{definition}
A local ring $(R, \fm)$ is called {\it $F$-rational} if it is a homomorphic image of a Cohen-Macaulay local ring and every parameter ideal is tight closed, i.e. $\fq^* = \fq$ for all $\fq$.
\end{definition}
\begin{definition}
A local ring $(R, \fm)$ is called {\it $F$-pure} if the Frobenius endomorphism $F: R \to R, x \mapsto x^p$ is a pure homomorphism. If $R$ is $F$-pure, then it is proved that every ideal $I$ of $R$ is Frobenius closed, i.e. $I^F = I$ for all $I$.
\end{definition}
\begin{definition}
\begin{enumerate}
\item A local ring $(R, \fm)$ is called {\it $F$-injective} if the Frobenius action on $H^i_{\fm}(R)$ is injective, i.e. $0^F_{H^i_{\fm}(R)} = 0$, for all $i \ge 0$.
\item  A local ring $(R, \fm)$ is called {\it $F$-nilpotent} if the Frobenius actions on all lower local cohomologies $H^i_{\fm}(R)$, $i \le d-1$, and $0^*_{H^d_{\fm}(R)}$ are nilpotent, i.e. $0^F_{H^i_{\fm}(R)} = H^i_{\fm}(R)$ for all $i \le d-1$ and $0^F_{H^d_{\fm}(R)} = 0^*_{H^d_{\fm}(R)}$.
\end{enumerate}
\end{definition}
\begin{remark}\label{F-sing}
\begin{enumerate}
\item It is well known that an equidimensional local ring $R$ is $F$-rational if and only if it is Cohen-Macaulay and $0^*_{H^d_{\fm}(R)} = 0$.
\item An excellent equidimensional local ring is $F$-rational if and only if it is both $F$-injective and $F$-nilpotent.
\item Suppose every parameter ideal of $R$ is Frobenius closed. Then $R$ is $F$-injective (cf. \cite[Main Theorem A]{QS17}). In particular, an $F$-pure ring is $F$-injective.
\item An excellent equidimensional local ring $R$ is $F$-nilpotent if and only if $\fq^* = \fq^F$ for every parameter ideal $\fq$ (cf. \cite[Theorem A]{PQ19}). 
\end{enumerate}
\end{remark}
\section{Proof of the main result}
This section is devoted to prove the main result of this paper. Without loss of generality we will assume that $R$ is complete with an infinite residue field. We recall Brian\c{c}on-Skoda's Theorem (cf. \cite[Theorem 5.6]{HH90}) that give a relation between the tight closure and the integral of an ideal.
\begin{theorem}[Brian\c{c}on-Skoda]\label{Tight closure briancon}
Let $R$ be a Noetherian ring of prime characteristic $p$, $I$ ideal generated by $n$ elements. Then for all $w\geq 0$ we have
$$\overline{I^{n+w}}\subseteq (I^{w+1})^*.$$
\end{theorem}
\begin{corollary}\label{HQ3.2.2} 
Keeping all assumptions of Theorem \ref{Tight closure briancon}, then for all $w\geq 0$ we have
$$\overline{I^{d+w}}\subseteq (I^{w+1})^*.$$
In particular, if $w=0$ then  
$$\overline{I^{d}}\subseteq I^*.$$
\end{corollary}
\begin{proof}
Let $J$ be a minimal reduction of $I$. We have $\mu(J)=\ell(I)\leq d$ (cf. \cite[Proposition 8.3.7 and Corollary 8.3.9]{SH}) and $\overline{I^k}=\overline{J^k}$ for every positive integer $k$. Applying Theorem \ref{Tight closure briancon} for $J$,  
$$\overline{I^{\ell(I)+w}}=\overline{J^{\ell(I)+w}}=\overline{J^{\mu(J)+w}}\subseteq (J^{w+1})^*\subseteq (I^{w+1})^*.$$
Thus, $\overline{I^{d+w}}\subseteq (I^{w+1})^*.$
\end{proof}
\begin{corollary}\label{briancon} Let $(R, \fm)$ be a Noetherian local ring of dimension $d$ and $\fq$ parameter ideal. We have
\begin{enumerate}[(1)]
\item If $R$ is excellent and $F$-nilpotent then $\overline{\fq^d} \subseteq \fq^F$.
\item If all associated prime ideals of $R$ are minimal then $\overline{\fq^{d+1}} \subseteq \fq^F$.
\end{enumerate}
\end{corollary}
\begin{proof} (1) By Corollary \ref{HQ3.2.2} we have $\overline{\fq^d} \subseteq \fq^*$. Moreover, $\fq^*=\fq^F$ since $R$ is $F$-nilpotent. Thus the first assertion holds.\\
(2) Take any $x\in \overline{\fq^{d+1}} $, By \cite[Theorem 6.8.12]{SH}, there exists $c\in R^{\circ}$ such that $cx^N\in \fq^{(d+1)N}$ for all large $N$, so $cx^N\subseteq cR\cap \fq^{(d+1)N}$. By Artin-Rees Lemma, there is $k\geq 1$ such that for large $N$ we have
$$cx^N\in cR\cap \fq^{(d+1)N}= \fq^{(d+1)N-k}(\fq^k\cap cR)\subseteq c\fq^{(d+1)N-k}.$$
 Because all associated prime ideals of $R$ are minimal,  $c$ is a non-divisior of zero so $x^N\in \fq^{(d+1)N-k}$ for large $N$. Hence, for large $N=p^e$ we have 
 $$x^{p^e}\in \fq^{(d+1)p^e-k}\subseteq \fq^{dp^e}\subseteq \fq^{[p^e]}.$$ 
Thus $x\in \fq^F$.
\end{proof}
We recall the concepts of type and socle of a module (cf. \cite[Section 1.2]{BH93}). 
Let $(R,\fm)$ be a Noetherian local ring, $M$ a $R$-finitely generated nonzero module with $\mathrm{depth}(M)=t$. Set $k=R/\fm$. Then the type of $M$ is 
$$\mathrm{r}(M):=\dim_k(\Ext_R^t(k,M)).$$ 
The socle of $M$ is $$\mathrm{Soc}(M):=(0:\fm)_M\cong \mathrm{Hom}_R(k,M).$$ If $\underline{x}$ is a maximal $M$-sequence then $\mathrm{r}(M)=\mathrm{dim}_k(\mathrm{Soc}(M/\underline{x}M))$.
\begin{lemma}\label{inequallength} 
Let $(R, \fm)$ be an Artinian local ring with $k=R/\fm$ infinite, and let $M$ be a finitely generated module over $R$ such that $\dim _k(\mathrm{Soc}(M))=s $. Khi đó $\ell_R(M)\leq s\ell_R(R)$.
\end{lemma}
\begin{proof}
Since $M$ is finitely generated over Artinian ring, $M$ is Artinian module and $\mathrm{Soc}(M)$ is an essential extension of $M$. We have
$$M\subseteq E_R(\text{Soc}(M)).$$
Set $k=R/\fm$ and $E=E_R(k)$. By Matlis duality, $E$ has finite length over $R$ and $\ell_R(E)=\ell_R(R)$.
Moreover, $\mathrm{dim}_k(\mathrm{Soc}(M))=s$ so $\text{Soc} (M)\cong k^s$ and
$E_R(\text{Soc}(M))\cong E^s$. Thus $M\subseteq E^s$ and
$$\ell_R(M)\leq s\ell_R(E)=s\ell_R(R).$$
The proof is complete.
\end{proof}
\begin{theorem}\label{ChanTrenBoiCM}
Let $(R, \fm)$ be a Cohen-Macaulay ring with prime characteristic $p$ of dimension $d$ with the embedding dimension $v$ and the type $s$. Set $Q = p^{\mathrm{Fte}(R)}$. Then we have
\begin{enumerate}[(1)]
\item 
$$e(R)\leq \systeme*{{(s+1)Q^{v-d}\binom{v-r-1}{r}}\quad\quad\quad  \text{ if }\dim (R)=2r+1,
{\frac{(s+1)}{2} Q^{v-d}\left(\binom{v-r}{r}+\binom{v-r-1}{r-1}\right)} \text{ if }\dim (R)=2r.
}
$$
\item If $R$ is $F$-nilpotent then
$$e(R)\leq\systeme*{{\frac{(s+1)}{2} Q^{v-d}\left({{v-r-1}\choose{r}}+{{v-r-2}\choose{r-1}}\right)} \text{ if }\dim (R)=2r+1 ,
{(s+1)Q^{v-d}{{v-r-1}\choose{r-1}}}\quad\quad\quad \text{ if }\dim (R)=2r.
}
$$
\end{enumerate}
\end{theorem}
\begin{proof}
Because the proofs of two assertions are almost the same, we will only prove (2)
Let $\fq = (x_1, \ldots, x_{d})$ be a minimal reduction of   $\fm$. Since $R$ is $F$-nipotent, $\mathrm{Fte}(R) < \infty$.  By Corollary \ref{briancon}(1), $\fm^{d} \subseteq \overline{\fm^{d}} = \overline{\fq^{d}} \subseteq \fq^F$. On the other hand, we have $(\fq^F)^{[Q]} = \fq^{[Q]}$. Thus $(\fm^{d})^{[Q]} \subseteq \fq^{[Q]}$.
Set $k=R/\fm$, $A=R/\fq^{[Q]}$, $\fn=\fm/\fq^{[Q]}$. Then $(A,\fn)$ is Artinian local ring and $(\fn^{d})^{[Q]}=0$. Let $l$ be arbitrary positive integer such that $ l\leq d$. We have $(\fn^{d-l})^{[Q]}\subseteq 0:_A(\fn^{l})^{[Q]}$ for all $ 1\le l\le d$. In other words, $(\fn^{d-l})^{[Q]}\subseteq\text{Ann}_A(\fn^{l})^{[Q]}$ and we can consider $(\fn^{d-l})^{[Q]}$  as a $A':=A/(\fn^{l})^{[Q]}$-module.
Moreover, $R$ is Cohen-Macaulay so $x_1,\ldots,x_{d}$ and $x_1^Q,\ldots,x_{d}^Q$ are maximal regular sequences. We have
$$\dim_k(\text{Soc}(\fn^{d-l})^{[Q]})\leq\dim_k(\text{Soc}(A))=\dim_k(\text{Soc} (R/\fq^{[Q]}))=r_R(R)=s.$$ 
The fisrt inequality due to $\text{Soc}(\fn^{d-l})^{[Q]}\subseteq \text{Soc}(A)$. By Lemma \ref{inequallength}, $\ell_A((\fn^{d-l})^{[Q]})=\ell_{A'}((\fn^{d-l})^{[Q]})\leq s\ell_{A'}(A')=s\ell_A(A').$ Thus,
\begin{align*}
e(R)=e(\fm)=e(\fq)=\frac{1}{Q^{d}} e(\fq^{[Q]})\le &\frac{1}{Q^{d}}\ell_R(R/\fq^{[Q]})\\
\leq &\frac{1}{Q^{d}}\ell_A(A)\\
\leq &\frac{1}{Q^{d}}(\ell_A((\fn^{d-l})^{[Q]})+\ell_A(A/(\fn^{d-l})^{[Q]}))\\
\leq&\frac{1}{Q^{d}}( s\ell_A(A/(\fn^{l})^{[Q]})+\ell_A(A/(\fn^{d-l})^{[Q]})).
 \end{align*}
Extend $x_1, \ldots, x_{d}$ to a minimal set of generators $x_1, \ldots, x_{d}, y_1, \ldots, y_{v-d}$ of $\fm$. Then $\bar{x}_1,\ldots,\bar{x}_{d}$, $\bar{y}_1, \ldots,\bar{y}_{v-d}$ is a set of generators of $\fn$. Now $A/(\fn^l)^{[Q]}$ is spanned by monomials
$$\bar{x}_1^{\alpha_1}\cdots \bar{x}_{d}^{\alpha_{d}}\bar{y}_1^{\beta_1Q + \gamma_1} \cdots \bar{y}_{v-d}^{\beta_{v-d}Q + \gamma_{v-d}},$$  
where $0 \le\alpha_1, \ldots, \alpha_{d}, \gamma_1, \ldots, \gamma_{v-d} < Q$ and $0 \le \beta_1 + \cdots + \beta_{v-d} < l$. The number of such monomials is less than or equal to $Q^{v} \binom{v-d+l-1}{l-1}$,
thus $$\ell_A(A/(\fn^l)^{[Q]}) \le Q^{v} \binom{v-d+l-1}{l-1}. $$
Similarly $$\ell_A(A/(\fn^{d-l})^{[Q]}) \le Q^{v} \binom{v-d+d-l-1}{d-l-1}. $$
So we have $$e(R) \le Q^{v-d}\left(s\binom{v-d+l-1}{l-1}+\binom{v-d+d-l-1}{d-l-1}\right).$$
Choosing $l=r$ if $d=2r$, choosing $l=r$ and $l=r+1$ if $d=2r+1$, we obtain that
$$e(R)\leq\systeme*{{\frac{(s+1)}{2} Q^{v-d}\left({{v-r-1}\choose{r}}+{{v-r-2}\choose{r-1}}\right)} \text{ if }\dim (R)=2r+1 ,
{(s+1)Q^{v-d}{{v-r-1}\choose{r-1}}}\quad\quad\quad \text{ if }\dim (R)=2r.
}
$$ 
The proof is complete.
\end{proof}
\begin{remark}
If $R$ is Gorenstein then $r(R)=1$, by Theorem \ref{ChanTrenBoiCM} we have the result of Huneke
and Watanabe \cite[Theorem 5.1]{HW15}.
\end{remark}
\begin{example}
Let $S=\mathbb{F}_p[X,Y]$ be a polynomal ring over $\mathbb{F}_p$ with prime $p$, $\frak{ m}=(X,Y)$ maximal ideal $S$, $f=X^aY^a$. Set $R=S/(f)S$, then $R$ is Gorenstein ring of dimension $d=1$, of embedding dimension $v=2$ with the type $r(R)=s=1$ and $H^0_{\frak{ m}}(R)=0$.\\ 
Next, we will find $\mathrm{Fte}(R)$.
The \v{C}ech complex $\check{C}(X, Y;S)$: 
$$0 \to S \xrightarrow{\quad} S_{X}\oplus S_{Y}\xrightarrow{\phi} S_{XY}\to 0,$$
where $\phi(u,v)=v-u$. For simplyfication we also use $u$ and $v$ to denote their images in the localizations of $R$ at $u$ and $v$ respectively.\\
The set of the exponents of all monomials of $S_{X}$ is $$\{(s,r)\mid s,r\in \mathbb{Z}, r\geq 0\}.$$
The set of the exponents of all monomials of $S_{X}\oplus S_{Y}$ is
 $$\{(s,r)\mid s,r\in \mathbb{Z};  s\geq 0 \text{ or } r\geq 0\}.$$
The set of the exponents of all monomials of $S_{XY}$ is $\{(s,r)\mid s,r\in \mathbb{Z}\}$. 
Then $$H^2_{\frak{ m}}(S)=S_{XY}/\mathrm{Im} (\phi)=\oplus_{s,r< 0} \mathbb{F}_pX^{s}Y^{r}.$$
Thus, $(H^2_{\frak{ m}}(S))_{(s,t)}\neq 0$ if and only if $s<0$ and $t<0$.
From an exact sequence 
$$0\longrightarrow S_1:=S(-a,-a) \xrightarrow{.f}S\longrightarrow R\longrightarrow 0,$$
induces the following exact
$$0=H^1_{\frak{ m}}(S)\longrightarrow H^1_{\frak{ m}}(R)\longrightarrow H^2_{\frak{ m}}(S_1)\xrightarrow{\psi} H^2_{\frak{ m}}(S)\to 0,$$ 
where $\psi$ is the mutiplication by $f$.\\
We have $H^1_{\frak{ m}}(R)=\mathrm{Ker} (\psi)$ has the set of exponents $E$ defined as follows
$$E=\{(s,r)\mid s,r<a; s\geq 0 \text{ or } r\geq 0\}.$$
($E$ is the colouring area between angle $\widehat{x'My'}$ and angle $\widehat{x'Oy'}$ including ray $Ox'$ and ray $Oy'$.) 
\begin{center}
\includegraphics[scale=0.8]{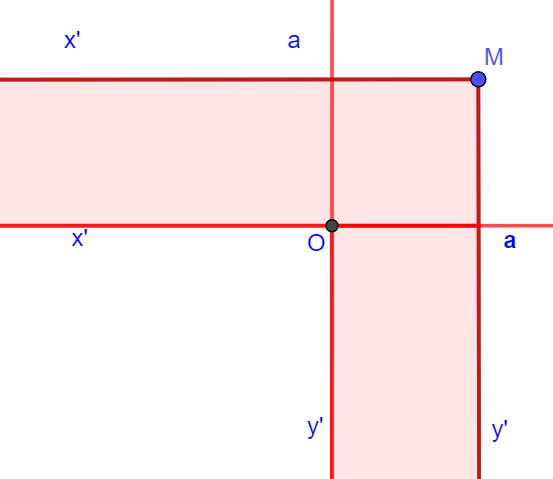}
\end{center}
The e-th Frobenius map $F^e: H^1_{\frak{ m}}(R)\to H^1_{\frak{ m}}(R) $ corresponds to a homothety on set $E$ of center $O$ and ratio $k=p^e$. Then the set of all exponents of $0^F_{H^1_{\frak{ m}}(R)}$ is
$$E_1=\left( E \setminus (Ox'\cup Oy')\right)\cup {O}.$$
So $\mathrm{Fte}(R)=\mathrm{HSL}(R)=\lceil \log_p(a) \rceil$, where $\lceil u \rceil$ is minimal integer such that greater than or equal to $u$. 
If we choose $p=a=2$, then the Hilbert seri of $R$ is
$$H_R(t)=\frac{1+t+t^2+t^3}{1-t}.$$
Thus, $e(R)=Q(1)=4$ where $Q(t)=1+t+t^2+t^3$ (cf. \cite[Section 6.1.1]{HH11}).
We have the equality in Theorem \ref{ChanTrenBoiCM}(1).
\end{example}

\end{document}